\documentclass[12pt]{amsart}
\usepackage{Baskervaldx}
\usepackage[baskervaldx,cmintegrals,bigdelims,vvarbb]{newtxmath}
\usepackage[cal=cm]{mathalfa}
\usepackage{enumerate, geometry, pgfplots, graphicx, subcaption, amsmath, array, pdfsync, tikz, tikz-cd, color, url, float, wrapfig, comment, MnSymbol,adjustbox}
\usepackage{comment}
\usepackage{xypic}
\usetikzlibrary{positioning}

\usepackage[hyperfootnotes=false, pdfencoding=auto]{hyperref}
\hypersetup{
  colorlinks= true,%Colors links instead of ugly boxes
  linkcolor=[RGB]{0,89,42},
  urlcolor=[RGB]{0,89,42},
  citecolor=[RGB]{0,89,42}
}

\usepackage[font=footnotesize]{caption}

\usetikzlibrary{calc}

\numberwithin{equation}{section}

\definecolor{mycolor}{RGB}{146, 214, 203}
\definecolor{myothercolor}{RGB}{179, 215, 232}

\newtheorem{theorem}{Theorem}

\newtheorem{proposition}[theorem]{Proposition}
\numberwithin{theorem}{section}

\theoremstyle{definition}

\newcommand{\hr}[2]{\hyperref[#1]{#2}}

\theoremstyle{definition}
\newtheorem{remark}[theorem]{Remark}

\newtheorem{question}[theorem]{Question}
\newtheorem*{question*}{Question}

\newtheorem{definition}[theorem]{Definition}

\def\AA{{\mathbb A}}
\def\NN{{\mathbb N}}

\def\LL{{\mathbb L}}

\def\PP{{\mathbb{P}}}

\def\CCal{{\mathcal{C}}}

\def\Lc{{\mathcal{L}}}

\def\id{{\mathrm{id}}}

\def\dra{{\dashrightarrow}}

\newcommand{\Hilb}{\mathrm{Hilb}}

\def \DC{{\mathrm{DomCor}}}
\newcommand{\OO}{\mathcal{O}}
\newcommand{\cd}{\textrm{corr.deg}}

\newcommand*{\sheafTor}{\mathcal{T}\kern -.5pt or}
\newcommand*{\Extc}{\mathcal{E}\kern -.5pt xt}
\newcommand{\irr}{\textrm{irr}}

\pgfplotsset{compat=1.15}

\begin{document}

\title{Slicing correspondences with high degree hypersurfaces}
\author{Ishan Banerjee}
\maketitle

\begin{abstract}
    We approximately compute the correspondence degree (as defined by Lazarsfeld and Martin) between two unbalanced complete intersections. This is accomplished by showing that the procedure of taking a subvariety of a product $Y \times Y'$ and intersecting it with $X \times Y'$ (for $X$ a sufficiently ample smooth divisor in $Y$) induces a bijection between two sets of varieties. This may be of independent interest.
\end{abstract}

\section{Introduction}
In the paper \cite{LM}, Lazarsfeld and Martin introduce the notion of the \emph{correspondence degree} between two varieties. Given two varieties $Y$ and $Y'$ of the same dimension, they define the correspondence degree between $Y$ and $Y'$ denoted, $\cd(Y, Y')$ by the formula: $$ \cd(Y, Y') = \min\{ \deg(W \to Y) \deg(W \to Y') \},
$$
where $W$ is ranges over all irreducible subvarieties of $Y \times Y'$ of the same dimension as $Y$ such that the projection maps from $W$ to $Y$ and $Y'$ are dominant.

They prove the following Theorem: 
\begin{theorem}
    Let $Y, Y'$ be very general hypersurfaces of degree $d, e \ge  2n+ 2$ in $\PP^{n+1}.$ Then $\cd (Y, Y') =  \irr (Y)\irr(Y') = (d-1)(e-1),$ where $\irr(Y), \irr(Y')$ are the degrees of irrationality of $Y$ and $Y'$ respectively.
\end{theorem}

They prove a bit more than this: the minimizing correspondence arises as a fiber product of the minimal coverings of projective space. 

In the same paper they ask the following question:
\begin{question}\label{qnLM}
Let $M$ be a smooth projective variety and $A$ an ample divisor on $M$. Let $d, d' \gg 0.$
    and $Y, Y'$ be  very general hypersurfaces in $M$ such that $Y \in |dA|, Y' \in |d' A|.$ Can one find conditions on $M, A$ such that $\cd(Y,Y') \approx \irr(Y) \irr(Y')?$
\end{question}

We note that the results of \cite{LUS} imply that $\irr(Y) = d \textrm{mfd}(M,A)$, where $\textrm{mfd}(M,A)$ refers to the \emph{minimum fibering degree} of $M$ with respect to $A$,  a quantity defined  and studied in the paper \cite{LUS}.

In view of this, Question \ref{qnLM} becomes equivalent to showing that   $$\cd(Y,Y') \approx \textrm{mfd}(M,A)^2 dd'.$$

We prove:
\begin{theorem}\label{hypvers}
    Let $M$ be a smooth projective variety and $A$ an ample divisor on $M$. Let $d \gg d' \gg 0.$ Let $Y, Y'$ be  smooth hypersurfaces in $M$ such that $Y \in |dA|, Y' \in |d' A|.$
    Then $$\cd(Y,Y') \approx Kdd'$$ for a fixed constant $K$ depending only on $M$ and $A$.
\end{theorem}
This can be seen as an answer to a corrected version of the above question of Lazarsfeld and Martin under the simplifying assumption that $d \gg d'.$
We can in fact prove much more:

\begin{theorem}\label{unbancorrdeg}
     Let $n,r \ge 1$.
     Let  $d_1 \gg d_1' \gg \dots d_r \gg d_r' \gg e, e' \ge 0$.
     Let $M, M'$ be smooth projective varieties of dimension $n$. Let $H, H'$ be very ample divisorson $M, M'$ respectively. Let $Y, Y'$ be complete intersections of codimension $r$ in $M, M'$, and let  the multidegrees of $Y$ and $Y'$  (with respect to $H$ and $H'$) be $(d_1, d_2 \dots, d_r)$ and $(d_1' ,\dots d_r')$ respectively. Assume furthermore, that we can find a chain of smooth complete intersections $Y \subseteq Y_1 \subseteq Y_2 \dots \subseteq Y_r = M $, and $Y' \subseteq Y'_1 \subseteq Y'_2 \dots \subseteq Y'_r = M' $ where $Y_i$ has multidegree $(d_{i+1}, \dots d_r)$ and $Y_i'$ has multidegree $(d_{i+1}', \dots d_r')$

     Then for some constant $K$, only depending on $M, M', A, A'$, $$\cd(Y,Y') \approx d_1 d_2 \dots d_r d_1' \dots d_r' K .$$
\end{theorem}

Let us note that the constant $K$, obtained in Theorem \ref{hypvers} is \emph{not}  equal to $\textrm{mfd} (M,A)^2$ for all choices of $(M,A)$, which is what it would be if $\cd(Y,Y') \approx  \irr(Y)\irr(Y').$ In the particular case when $M = \PP^n$ and $A$ is a hyperplane section,  our constant $K$ does agree with $\textrm{mfd}(M,A)^2$ (which is actually $1$ in that particular situation).

An explicit example where these constants differ is given in Section \ref{Prod}.

In fact, we can explain what this constant $K$ is in explicit geometric terms and show that the minimizing correspondence arises as a fiber product of sorts. 

Our degree assumptions in Theorem \ref{unbancorrdeg} and \ref{hypvers} should be seen as analogues of the \emph{unbalanced} assumptions  in \cite{LUS}. Indeed, the main motivation of the authors in starting this project was to try to generalize the ideas of \cite{LUS} to the context of correspondence degree, and to thereby  solve several  problems in this subject in the unbalanced situation. We essentially succeed in doing so and in the process we define several concepts that we think are of importance to understanding correspondences between hypersurfaces and complete intersections. 

We can  prove all of our results over an arbitrary algebraically closed field of characteristic 0. However, there are some serious issues regarding separability that arise when trying to generalize the results to positive characteristic, we will briefly remark on the difficulties that arise in doing this in the final section of this paper.

Our results also have the following advantage: We do not assume that our hypersurfaces $Y, Y'$ are very general, which appears to be the only setting in which results have been proven so far.

Let us remark  on the \emph{unbalanced} assumption appearing in Theorems \ref{hypvers} and \ref{unbancorrdeg}, i.e. the assumption that not only the numbers appearing in our multidegrees are large, but that they are all far away from each other in a particular way. This is a very strong assumption, but it is necessary as Theorems \ref{hypvers} and \ref{unbancorrdeg}  are \emph{false} without this assumption. This can be seen with the following simple example: Let $d= d'$, $Y= Y'$. Then $\cd(Y,Y') =1$, which is certainly not $\approx K d^2$ for large $d$. However, it is certainly possible that one could recover the statement of Theorem \ref{hypvers} say with just the assumption that $d, d' \gg 0$, if we add the additional assumption that the pair $(Y, Y')$ is very general.
\begin{remark}
    In this paper we will often say that two functions $f, g : \NN^n \to \NN$ are approximately equal to each other, or $f(d_1 ,\dots, d_n) \approx g(d_1, \dots d_n)$. By this we mean that as $d_1, \dots, d_n$ goes to \` infinity\' (it will be clear what we mean by this in the usage) we have the limit $\frac{f(d_1, \dots, d_n)}{g (d_1 ,\dots, d_n )}  \to 1.$

    We will also oftern say that a function is $O(1)$ if it bounded by some constant independendent of the $d_i, d'_i$.
\end{remark}

\subsection{Organisation}
This paper is organised into the following parts:
\begin{enumerate}
    \item We first define the notion of \emph{dominant} $k$- correspondences and define two slicing operations on them which we call $\Phi$ and $\Psi$. We then state the main (in our opinion) Theorems of this paper, Theorems \ref{main} ,\ref{maink}, \ref{mainpsi} and \ref{mainpsik} which essentially say that these slicing operations define bijections. 
    \item The Theorems of the introduction, i.e. Theorem \ref{hypvers} and Theorem \ref{unbancorrdeg} are then shown to immediately follow  from these Theorems on slicing operations.
    \item 
    The special case when $Y$ is a product of two curves is then discussed. It is here that we establish that our constant $K$ is not in general equal to $\textrm{mfd}(Y,\Lc)^2$.
    \item We then deal with the some consequences of our results to covers of hypersurfaces.
    \item  Theorem \ref{main} is then proven. This is perhaps the heart of the paper, the Theorems \ref{maink}, \ref{mainpsi}, \ref{mainpsik} will be proven using Theorem \ref{main} in an essential way.
    \item We then prove that the slicing operations $\Phi$ and $\Psi$ can be defined in families.
    \item Finally, we  use Theorem \ref{main} in families to show that Theorems \ref{maink}, \ref{mainpsi} and \ref{mainpsik} are true.

    \item Finally we comment on what goes wrong with our arguments in positive characteristic.
\end{enumerate}

\subsection{Acknowledgements}
The author like to thank Iason Kountouridis. Will Newman and Olivier Martin for helpful comments on various drafts of this paper. 
This project started out as a joint endeavour with David Stapleton and a lot of the credit for the initial ideas of this project must be given to him. 
\section{The main definitions}
We work over an algebraically closed field of characteristic $0$. Varieties are assumed to be varieties over the field unless otherwise mentioned.

Let $Y, Y'$ be smooth projective varieties of dimension $n$. Let $H, H'$ be very ample line bundles on $Y, Y'$ respectively.

\begin{definition}
    Let $k \in \NN$.
    A \emph{prime dominant k correspondence} between $Y$ and $Y'$ is an irreducible subvariety $W \subseteq Y \times Y'$ of dimension $n +k$ surjecting onto both $Y$ and $Y'$.
\end{definition}

An prime dominant $0$ correspondence is just an irreducible variety $W \subseteq Y \times Y'$ dominating both factors and such that $\dim W = \dim Y = \dim Y'$, precisely the kind of object that shows up when defining correspondence degree.

Associated to such a prime dominant $k$-correspondence $W$ is a pair of numbers which we will call its degree. We define this as follows: Let $W \subseteq Y \times Y'$ be a prime dominant correspondence.  Let $p : Y \times Y' \to Y$ and $p' : Y \times Y' \to Y'$ be the projection maps. Let $y \in Y, y' \in Y'$ be  general points. 
We then define the degree of $W$ by the formula  $$  \deg W = ( p^*(y) \cdot p'^{*}((H')^k), p'^*(y') \cdot p^{*}((H)^k)).$$  

We now define a general dominant correspondence.
\begin{definition}
    Given $k \in \NN$,
    a dominant $ k$ correspondence $W$ between $Y$ and $Y'$ is an integral linear combination of prime dominant $k$ correspondences with positive coefficients, i.e. $W$ is a  formal sum of the form $\sum_i a_i W_i$ where $a_i \in \NN$ and $W_i$ are prime dominant $k$ correspondences. We define the degree of $W$ to be $\sum_i a_i \deg(W_i).$

    We define the set $\DC^k _{\le (a,b)} (Y, Y')$ to be the set of dominant $k$ correspondences $W$ between $Y$ and $Y'$ such that $\deg W = (a_0,b_0)$ where $a_0 \le a$ and $b_0 \le b$.
\end{definition}

Let $d, d' >0$. Let $X, X'$ be smooth  hypersurfaces in $Y$ and $Y'$ respectively such that $X\in |d H|$ and $|X'| \in |d' H'|.$

We will now need to define \emph{uneven} dominant correspondences. 

\begin{definition}
    Let $ Y, X'$ be smooth projective varieties of dimension $n+1$ and $n$ respectively. 
    Let $k \in \NN$.
    A \emph{prime uneven dominant k correspondence} on $Y \times X'$ is an irreducible subvariety $W \subseteq Y \times X'$ of dimension $n+k+1$ surjecting on to both factors $Y$ and $X'$ under the respective projection maps.     
\end{definition}

Let $W$ be an prime uneven dominant $k$ correspondence between $Y$ and $X'$. Let $y \in Y , x' \in X'$ be general points. We define the degree of $W$  by the formula:
$$  \deg W = (\# (\{y\} \times (H')^k \cap p^{-1}(y)), \# (H^{k+1} \times \{y'\} \cap p'^{-1}(y'))),$$  we note that these intersection multiplicities are independent of the choice of $y,x'$ as long as they are general.

\begin{definition}
    Let $k \in \NN$.
    An uneven dominant $k$ - correspondence between $Y$ and $X'$ is is a formal positive sum of prime dominant $k$ correspondences, i.e. $W$ is a  formal sum of the form $\sum_{i=1}^k a_i W_i$ where $a_i \in \NN$ and $W_i$ are prime dominant $k$ correspondences. We define the degree of $W$ to be $\sum a_i \deg(W_i).$

    We define the set $\DC^k _{\le (a,b)} (Y, X')$ to be the set of dominant correspondences $W$ between $Y$ and $X'$ such that $\deg W = (a_0,b_0)$ where $a_0 \le a$ and $b_0 \le b$. 
\end{definition}

We note that we use the same notation $\DC (\_,\_)$ to denote the uneven and regular dominant correspondences- we hope this will not cause the reader any confusion. To understand whether we refer to uneven dominant correspondences in such a situation we just need to look at whether the varieties appearing in the parentheses have the same dimension or not. We will also often refer to both regular and uneven dominant correspondences as dominant correspondences without any adjectives, again we hope this does not cause any confusion.

Given a dominant correspondence $W$ and $a,b \in \NN$, we will say that $\deg W \ge (a,b)$ if $\deg W = (a',b')$ where $a' \ge a$ and $b' \ge b$.

\subsection{The slicing operations $\Phi$ and $\Psi$}
 In this subsection we will define two functions between sets of dominant correspondences arising from slicing a correspondence with a hypersurface.

 Let $Y,X'$ be smooth projective varieties of dimension $n+1$ and $n$ respectively. Let $H, H'$ be very ample divisors on $Y, X'$ respectively. Let $d >0$ and $X$ be a smooth hypersurface in $Y$ belonging to $|dH|.$  Given an uneven dominant correspondence $W$ between $Y$ and $X'$, one can do the following operation: we first form the intersection between $W$ and $X \times X'$.
 Let $W_1 , \dots W_r$ denote all the irreducible components of $W \cap X \times X'$ that dominate both $X$ and $Y'$. We then define $\Phi(W) = W_1 \cup \dots \cup W_r$. If $\Phi(W)$ is nonempty , then it is clearly a dominant correspondence between $X$ and $X'.$

 One immediately sees that given $e_1, e_2 \in \NN$, this construction $\Phi$ defines a function $$\DC^k_{\le (e_1, e_2)}(Y, Z) \to \DC^k_{\le (e_1, de_2)} (X,Z),$$ again assuming that the variety $\Phi(W)$ is nonemepty. Note that the domain of the function $\Phi$ is a set of uneven dominant correspondences, while the co-domain is a set of regular dominant correspondences.

Our main Theorems with regard to $\Phi$ are as follows:

\begin{theorem}\label{main}
    Let $d, e,  e' \in \NN$. Let $d \gg e'$ (implicitly  we also assume that $d \gg (H')^n$ ).

Then the map $\Phi: \DC^0_{\le (e, e')}(Y, X') \to \DC^0_{\le (e, de')} (X,X') $ as defined above is a bijection.  

Furthermore $\Phi$ takes prime correspondences to prime correspondences.
\end{theorem}

\begin{theorem}\label{maink}
    Let $d \gg e'$ (implicitly  we also assume that $d \gg |H'|^n$ ). Let $k \ge 1$

Then the map $\Phi: \DC^k_{\le (e, e')}(Y, X') \to \DC^k_{\le (e, de')} (X,X') $ as defined above is a bijection.  

Furthermore $\Phi$ takes prime correspondences to prime correspondences.
\end{theorem}

We note that we have singled out the case $k=0$ into Theorem \ref{main}. This is because we will very heavily use Theorem \ref{main} in proving both Theorem \ref{maink} and several of our other Theorems and we thought it would be more sensible to make it a separate statement. 

Let us now define another slicing operation, which we will denote by $\Psi$.

Let $Y, Y'$ be smooth projective varieties of dimension $n+1$. Let $H,H'$ be very ample divisors on $Y,Y'$ respectively. Let $d' >0$. Let $X'$ be a smooth hypersurface in $Y'$ belonging to $|d'H'|$. Let $W$ be a dominant correspondence between $Y$ and $Y'$. Then we can perform the following operation: we first form the intersection $W \cap (Y \times X')$. We define $\Psi(W)$ to be the union of all components of $W \cap (Y \times X')$ dominating both $Y$ and $X'$.The variety $\Psi(W)$, if nonemepty, is an uneven dominant correspondence between $Y$ and $X'$

We note now that given $k,e,e' \in \NN$, and assuming that the varieties $\Psi(W)$ are nonempty,  $\Psi$ actually defines a function $$\Psi : \DC^{k+1}_{\le (e,e')} (Y,Y') \to \DC^{k}_{\le (d'e,e')} (Y, X').$$ 
Let us now state our main Theorems for this function $\Psi.$
\begin{theorem}\label{mainpsi}
Let $e,e' \in \NN$
    For $d'\gg e$, the map  $$\Psi: \DC^1_{\le (e,e')} (Y,Y') \to \DC^{0}_{\le (d'e,e')} (Y, X')$$  is a bijection.
    Furthermore $\Psi$ takes prime correspondences to prime correspondences.
\end{theorem}

\begin{theorem}\label{mainpsik}
Let $k ,e, e' \in \NN$.
For $d'\gg e$, the map  $$\Psi: \DC^{k+1}_{\le (e,e')} (Y,Y') \to \DC^{k}_{\le (d'e,e')} (Y, X')$$  is a bijection.
Furthermore $\Psi$ takes irreducible correspondences to irreducible correspondences.
\end{theorem}

We single out the case $k=0$ as its own Theorem for similar reasons to those for singling out Theorem \ref{main}.
\section{Proof of Theorem \ref{unbancorrdeg}}
In this section we prove Theorem \ref{unbancorrdeg} assuming that Theorems \ref{main}, \ref{maink}, \ref{mainpsi} and \ref{mainpsik} are true.
We begin with the following proposition.
\begin{proposition}\label{unban}
    Let $n,r \ge 1$.
    Let $d_1 \gg d_1' \gg \dots d_r \gg d_r' \gg e, e' >0$.
    Let $Y, Y'$ be smooth projective varieties of dimension $n$. with very ample divisors $H, H'$  respectively. Let $X, X'$ be complete intersections of codimension $r$ in $Y, Y'$, and let  the multidegrees of $X$ and $X'$ be $(d_1, d_2 \dots, d_r)$ and $(d_1' ,\dots d_r').$ Assume furthermore that we can find a chain of smooth complete intersections $X \subseteq X_1 \subseteq X_2 \dots \subseteq X_r = Y $, and $X' \subseteq X'_1 \subseteq X'_2 \dots \subseteq X'_r = Y' $ where $X_i$ has multidegree $(d_{i+1}, \dots d_r)$ and $X_i'$ has multidegree $(d_{i+1}', \dots d_r')$

     Then any dominant 0 correspondence $W$ between $X$ and $X'$ of degree $ \le (d_1' \dots d_r' e, d_1 \dots d_r e)$ is obtained from slicing an $r$ correspondence of degree $\le (e,e')$ between $Y$ and $Y'$, more explicitly there is a dominant  $r$ correspondence $W_r$, of appropriate degree such that $(\Psi \circ \Phi)^r (W_r) = W$ .

    As a special case, if $X,X'$ are curves, then any such $W$ is obtained by slicing a divisor in $Y \times Y'$.
\end{proposition}
\begin{proof}
    Consider the two flags of complete intersections $X \subseteq X_1 \subseteq X_2 \dots \subseteq X_r = Y $, and $X' \subseteq X'_1 \subseteq X'_2 \dots \subseteq X'_r = Y' $. We may iteratively apply Theorems \ref{maink} and \ref{mainpsik} to lift $W$ to a flag of  correspondences $ W \subseteq W_1 \subseteq \dots \subseteq  W_r$ where $W_i$ is an $i$ correspondence between $X_i$ and $X_i'$ of appropriate multidegree, where for each $i$ the dominant correspondence $W_i$ is obtained as $\Psi (\Phi(W_{i+1})).$
    
     The variety $W_r$ is then our required correspondence, we obtain the required bounds on its degree from the bounds in Theorems \ref{maink} and \ref{mainpsik}.
\end{proof}

We will also need a  lower bound on the degrees of the correspondences $\Phi (W_1),\Psi(W_2)$ in terms of the degrees of $W_1$ and $W_2.$

\begin{proposition}\label{degreelowerbound}

    \begin{enumerate}
        \item Let $W \in \DC^k(Y, X')$ be such that $\deg W = (e,e').$ Assume that $d \gg e',$ so that we are in the context where we can apply Theorem \ref{maink}.
        Then $\deg \Phi(W) \ge (e, (d-1)e') \approx (e,de')$.
        \item  Let $W \in \DC^k(Y, Y')$ be such that $\deg W = (e,e').$ Assume that $d' \gg e,$ so that we are in the context where we can apply Theorem \ref{maink}.
        Then  $\deg \Psi(W) \ge  ((d'-1)e, e')\approx (d'e,e')$
    \end{enumerate}
\end{proposition}
\begin{proof}
    We will only prove (1) since (2) is almost identical. We note that by Theorem \ref{maink} the map $\Phi$ is injective and that $\deg \Phi(W) \le (e, de')$. If the degree of $\Phi(W) \le (e, (d-1) e')$ then by the surjectivity part of Theorem \ref{maink} we may conclude that there is some $W' \in \DC^k(Y, X')$ such that $\deg W' \le (e, e' -1)$ and $\Phi(W') = \Phi(W).$ But the injectivity of $\Phi$  implies that $W' = W,$ which is a contradiction since they have different degrees.
\end{proof}
We now prove Theorem \ref{unbancorrdeg}.

\begin{proof}[Proof of Theorem \ref{unbancorrdeg}]
    Let $$K = \min \{ab| \deg W = (a,b), W\textrm { is a dominant } r \textrm{ correspondence between } M \textrm {and }M'\}.$$

    Let $W_0$ be a dominant $r-$correspondence between $M$ and $M'$ realizing this minimum. Let $W_1 = \Phi \circ \Psi (W), W_2 = \Phi \circ \Psi (W_1), \dots W_r = \Phi \circ \Psi (W_{r-1}),$ where  the slicing operations $\Phi$ and $\Psi$ are with respect to the inclusions in the flags $Y \subseteq Y_1 \subseteq \dots \subseteq M$ and $Y' \subseteq Y'_1 \subseteq \dots \subseteq M'.$  As a consequence of Proposition \ref{degreelowerbound}, the degree of $W_r$ is approximately $(d_1' \dots d_r' a,d_1d_2 \dots d_r b).$ Thus to finish the proof of Theorem \ref{unbancorrdeg} it suffices to establish that there are no correspondences between $Y$ and $Y'$ of significantly smaller product degree.

     Assume for the sake of contradiction that $W$ is a correspondence between $X$ and $X'$ of degree $(a,b)$ where $$ab < d_1 d_2 \dots d_r d_1' \dots d_r' K .$$ Under our assumptions the above inequality implies that  $d_1 \gg \frac{b}{d_1}$.
      Thus by Theorem \ref{maink} (and Proposition \ref{degreelowerbound}) we may lift $W$ to a dominant correspondence $W'$ between $X_1$ and $X'$ of degree $(a, \frac{b}{d_1}).$ We then note that $d_1' \gg \frac{a}{d_1'}$.
       We then apply Theorem \ref{mainpsik} to conclude that we may again lift $W'$ to obtain a dominant $1$ correspondence between $X_1$ and $X_1'$ of degree $(\frac{a}{d_1'}, \frac{b}{d_1}).$ Repeating this procedure we eventually end up with a dominant correspondence of degree $(\frac{a}{d_1' \dots d_r'}, \frac{b}{d_1,\dots d_r})$ between $Y$ and $Y'$. However the fact that $K$ is the minimal product degree of a dominant $r$ correspondence between $Y$ and $Y'$, then implies that $ab \ge d_1 d_2 \dots d_r, d_1' \dots d_r' K $ which is a contradiction. This establishes the result.
\end{proof}
\section{Some applications and examples} \label{Prod}
\subsection{Product of two curves}
In this subsection we will consider the case when $Y = Y' = C_1 \times C_2$ where $C_1$ and $C_2$ are curves. We will assume the very ample divisor $H$ on $Y$ is of the form $k( C_1 \times \{p\} + \{q\} \times C_2)$ for some $k >0$.

We note that in this situation there is a dominant $1$ correspondence $W \subseteq Y \times Y$ that is simply the preimage of the diagonal under the map $Y\times Y\to C_1 \times C_1$. We note that the degree of $W$ is $(k,k)$.

Let us note that $mfd(Y, H)$ can also be computed; at least after making some assumptions on $C_1$ and $C_2$. Let $f: Y \dra \AA^1$ be a rational function that does not factor through the projection to $C_1$. Then when we restrict $f$ to $ \{p\} \times C_2$ it has degree at least $\irr(C_2).$ Thus the fibering degree of $f$ with respect to $\Lc$ is at least $k \times \irr(C_2).$ Consider a rational map $C_1 \dra \AA^1$ realising the gonality. The composite of this map with the projection $Y \to C_1$ is a rational map of fibering degree $k \times \irr(C_1)$. Assume that $\irr(C_1) <\irr(C_2).$ In this case the above shows that $mfd(Y, H) = k \irr(C_1).$

This gives us an explicit example of where  the constant suggested by Lazarsfeld- Martin and our constant differ, ours is at most $k^2,$ whereas their constant is exactly $k^2 \irr(C_1)^2$.

Let us remark on why this difference in constants occurs and how we came up with this example. 

Let $Y$ be a smooth projective surface and $\LL$ a very ample line bundle. An easy way to construct a dominant $1-$ correspondence contained in $Y \times Y$ is as follows: take a  dominant rational function $f: Y \dra \AA^1$ that realises $mfd(Y,H)$. We then pullback the diagonal $\Delta \subseteq \AA^1 \times \AA^1$ via the function $f \times f$ to obtain a dominant $1-$ correspondence $W$ of degree $(mfd(Y, H), mfd(Y, H))$.

Morally speaking, the Lazarsfeld Martin constant is correct if the smallest degree 1 correspondence on $Y$ is of this form. 

However as we have just seen,this need not be the case. In the case of a product of two curves, the product structure allows us to naturally construct a dominant $1-$ correspondence. More generally, let $Y$ be a smooth projective variety of dimension $n+1$ with a dominant map to an variety $Z$ of dimension $n$. Then we can pull back the diagonal $\Delta \subseteq Z \times Z$ to get a dominant $1$ correspondence in $Y \times Y$ and in the case when $Y$ is a product of curves, this gives us a smaller dominant correspondence.

\subsection{Covers of hypersurfaces}
In this subsection we will discuss a curious applcation of out results to the setting of covers of hypersurfaces.
Let $Y,Y'$ be smooth projective varieties of dimension $n+1$, let $H, H'$ be very ample divisors and let $d \gg d' \gg e.$ Let $X'$ be  a smooth hypersurface in $Y'$ in the linear system $ |d' H'|$.
A consequence of our results that we would like to emphasize is that under our unbalanced assumptions, $X'$ is not covered by \emph{any} smooth projective hypersurface $X$ in $|d H|$ for $d\gg d'$ unless the covering degree is greater than $e$.

For example, suppose we consider the situation when $n=1$. Then in the above situation, $Y, Y'$ are surfaces and $X, X'$ are curves. There is no covering $X \to X'$ of degree $\le e.$ This is quite surprising, as $X'$ is surely covered by some curves with the same genus as $X$ and also our statements are for \emph{all} smooth curves in the linear system $|d'H'|$ and not just for the general curve. It is not clear to us if this consequence of our Theorem has a more elementary proof.

\section{Proof of Theorem \ref{main}}

In this section, we will prove Theorem \ref{main}.

\subsection{Review of Cayley-Bacharach sets}
Let us now review the definition of a Cayley Bacharach set as well as discuss briefly where these objects show up in the course of this paper. 

\begin{definition}
    Let $X$ be a variety. Let $\Lc$ be a line bundle on $X$. Let $Z \subseteq X$ be a reduced set of points. We say $Z$ is a Cayley-Bacharach set for $\Lc$ if for any subset $Z' \subseteq Z$ containing all but one point of $Z$, and any $f \in H^0(X, \Lc)$ vanishing on $Z'$, the section $f$ vanishes on $Z$ as well.
\end{definition}

Some important facts about Cayley-Bacharach sets that we will use in our proofs are as follows:
\begin{enumerate}
    \item Let $d >0$. Let $X = \PP^n$ and $\Lc = \OO(d)$. Then if $Z \subseteq \PP^n$ is  a Cayley Bacharach set for $\Lc$, then it contains at least $d$ points.
    \item Let $d \gg e >0$.   Let $X = \PP^n$ and $\Lc = \OO(d)$. If $Z \subseteq \PP^n$ is  a Cayley Bacharach set for $\Lc$ and $|Z| < ed- f(e)$ (for a particular function $f$ not depending on $d$), then $Z$ lies on a curve $C$ of degree $\le e$, Furthermore there is a unique curve of minimal degree containing $Z$(this is Theorem 1.13 of \cite{Ban}).
\end{enumerate}

Let us sketch out how Cayley Bacharach sets and the above facts about them are used in this section.

Let $Y$ be a smooth projective variety of dimension $n+1$, $\Lc$ a very ample line bundle on $Y$ and let $X$ be a smooth hypersurface in $Y$ in the linear system associated to $\Lc^d.$ Let $e, e' > 0$ and assume that $d \gg e'. $Let $W \in \DC^0_{\le (e,e')} (X, \PP^n)$. Then for general $z \in \PP^n$, the set $W_z = W \cap X \times \{z\}$ is Cayley Bacharach for $K_X$. Now the embedding of $Y$ into $\PP (H^0(Y ,\Lc) ^{\vee}) = \PP^N$ gives us an embeddding of $W_z$ in $\PP^N$ as a Cayley Bacharach set- this enables us to conclude that it lies on a curve $C_z \subseteq \PP^N$ of degree $\le e'$. Some further arguments enable us to prove that $C_z$ is contained in $Y$. We then construct a variety $\tilde W$ that is morally the union  $\cup_{z \in \PP^n} C_z \times \{z\}$ and prove that $\Phi (\tilde W) = W.$ To prove these things we crucially use the fact that these curves $C_z$  both exist and are unique  (after some assumptions).

\subsection{Well-definedness}
Let us begin by noting that it is not even clear that if $W$ is in $\DC(Y,X'), $ $\Phi(W)$ is well defined. We defined $\Phi(W)$ by intersecting $W$ with $X \times X'$ and throwing out components that do not dominate $X$ and $X'$. It is conceivable that \emph{no} component of $ W \cap (X \times X')$ dominates both $X$ and $X'$. In this subsection we shall show that this can not happen under the assumptions of Theorem \ref{main}.

Let $W \subseteq Y \times X'$ be a prime dominant uneven $0$-correspondence of degree $(e,e')$.  Then we may birationally modify $W, Y, X'$ as follows:
    Given $x \in X',$ let $W_x =  W \cap Y \times\{x\}.$
    Let $\Hilb_{e'} Y$ denote the Hilbert scheme of $H-$degree $e'$ curves in $Y$.  $W$ defines a rational map $X'\dra \Hilb_{e'} Y$, defined by $x \mapsto W_x$. Let 
    \[
\xymatrix{
\bar X' \ar[d]\ar[rd]\\
X' \ar@{-->}[r] & \Hilb_{e'} Y
}
\]
be a resolution of the map $ X' \dra \Hilb_{e'} Y.$
    We define $\bar W$ to be the strict transform of $W$ in $ Y \times \bar X'$. 
We note that
      the projection map   $\bar W \to W$ is birational essentially by the definition  of strict transform. Furthermore for \emph {any} $x \in \bar X'$ , $\bar W \cap Y \times \{x\}$ is a curve of degree $e'$ and the curve is reduced for general $x$. 

     Let $\bar W_X$ denote $\bar W \cap (X \times \bar X')$. Note that $\bar W_X$ need not be irreducible and furthermore it is not a priori clear that the components of $\bar W_X$  dominate $X$ or $\bar X'.$

    \begin{proposition}\label{domboth}
        Assume that $d \gg e'$. Let $W_X^{\circ}$ denote the union of all components of $\bar W_X$ that surject onto $X$. Then every component of  $W_X^{\circ}$  surjects onto $\bar X'$ as well.
    \end{proposition}
    \begin{proof}
        We begin by noting that since $ \bar W\to Y$ is surjective, so is $\bar W_X \to X$ and thus $W_X^{\circ}$ is nonempty. Let $U$ be a component of $W_X^{\circ}$. Let us assume for the sake of contradiction that the map $U \to \bar X'$ is not surjective. In that case, let $V \subseteq \bar X'$ be the image of $U$. The general fiber of the map $U \to V$ is at most $1$ dimensional. Furthermore if the map $U \to V$ were quasifinite, then the dimension of $V$ would equal that of $\bar X'$, which contradicts the irreducibility of $\bar X'$. Thus we may assume that the map $\bar U \to V$ has 1 dimensional fibers. However these fibers must then be curves of degree $\le e'$ and $U$ is covered by curves of degree $\le e'$. This then implies that $\textrm{covdeg} (X) \le \textrm{covdeg} (U) \le e'.$ However, by Lemma 3.3  of \cite{LUS} and Lemma 1.3 of \cite{BPELU}  $\textrm{covdeg}(X) \ge d - O(1)$.  So for $d \gg 0$ this is not possible. Thus $U$ is forced to surject on to $\bar X'$, which establishes the result.
    \end{proof}

    Let $f: \bar X' \to X'$ be the defining map. We note that the fact that $f$ is birational immediately tells us  that $(\id \times f)(W_X^{\circ})$ is now a dominant correspondence between $X$ and $X'$. 

        \begin{proposition}
   We have the equality
        $$\Phi(W) = (\id \times f) (W_X^{\circ}).$$
    \end{proposition}

    \begin{proof}
Let $U$ be a component of $\Phi(W).$ We shall establish that $U \subseteq (\id \times f)(W_X^{\circ}).$ Let $\bar U$ be the strict transform of $U$ under the map $(\id \times f)$. Then, $\bar U \subseteq \bar W,$ and hence $\bar U \subseteq \bar W_X.$ Since, $\bar U$ dominates $X$ and $\bar X'$ (since $U$ is a component of $\Phi(W)$), $\bar U \subseteq W_X^{\circ}$. This in turn implies that $U \subseteq (\id \times f)(W_X^{\circ}).$

Let $V$ be a component of $W_X^{\circ}$. Then, since $V$ dominates $X$ and $\bar X'$, $(\id \times f) (V)$ dominates $X$ and $X'$, and hence $(\id \times f) (V)$ is a component of $\Phi(W)$. Thus, we have the required equality.
    \end{proof}

As a result of the above Propositions, given $W \in \DC^0_{\le(e,e')} (Y,X')$ the variety $\Phi(W)$ is indeed in the set $\DC^0_{\le(e,de')}(X, X')$

\subsection{Reducing to $H^0(X', K_{X'}) =0$}

To establish Theorem \ref{main}, we will reduce it to the following special case-
\begin{theorem}\label{mainspecialcase}
    Theorem \ref{main} is true if $H^0(X',  K_{X'}) =0$.
\end{theorem}
In this subsection we will establish Theorem \ref{main} assuming \ref{mainspecialcase}, though we only really require Theorem \ref{mainspecialcase} in the case when $X' = \PP^n$. In the  two subsections following this one, we will establish Theorem \ref{mainspecialcase}.

\begin{proposition}\label{redpn}
 Assume that Theorem \ref{mainspecialcase} is true. Then for $d \gg e'$ ,
the map $$\Phi: \DC^0_{\le (e,e')} (Y,X') \to \DC^0_{\le (e,de')} (X,X')$$ is injective.  
\end{proposition}
\begin{proof}
    Let $W^1, W^2 \in \DC_{\le(e,e')} (Y, X')$. Let us assume that $\Phi(W^1) = \Phi(W^2).$ We shall establish that $W^1 = W^2.$ We shall assume that $W^1$ and $W^2$ are irreducible, as the result immediately reduces to this case.
    Given $X', W^1$ and $ W^2,$ we fix once and for all two dominant morphisms $f_1, f_2: X' \to \PP^n$ such that $f_i^* \OO(1) = H'$ and  the product $f_1 \times f_2$ is a birational map between $X'$ and its image $\bar X' \subseteq \PP^n \times \PP^n$. The maps  $f_1, f_2$ are assumed to be general.
    
    We note that for all relevant values of $i$ and $j$,  the variety $(\id \times f_i)(W^j)$ is an uneven dominant correspondence between $Y$ and $\PP^n$ of degree $\le (e, k e')$, where $k$ is the maximum of the degrees of  $f_1$  and $f_2$.
    
    If $\Phi(W^1) = \Phi(W^2)$, then $(\id \times f_i)(\Phi(W^1)) = (\id \times f_i)(\Phi(W^2))$. We note here that for all relevant values of $i,j$, $(\id \times f_i)(\Phi(W^j))$ is a dominant correspondence between $Y$ and $\PP^n$. We further note that $\Phi((\id \times f^i) (W^j)) = (\id \times f^i)(\Phi(W^j)).$
    
    By Theorem \ref{mainspecialcase}, this implies $(\id \times f_i)(W^1) = (\id \times f_i)(W^2)$. This then implies $(\id \times f_1) \times_Y (\id \times f_2)(W^1) = (\id \times f_1) \times_Y (\id \times f_2)(W^2)$.
    
    We recall that a point of $(\id \times f_1) \times_Y (\id \times f_2)(W^i)$ is of the form $(y, p_1, p_2)$ such that there is some $(y, x') \in W^i$ satisfying $p_1 = f_1(x')$ and $p_2= f_2(x').$ Since $f_1 \times f_2$ is birational onto its image, the equality $$(\id \times f_1) \times_Y (\id \times f_2)(W^1) = (\id \times f_1) \times_Y (\id \times f_2)(W^2)$$ implies that  $W^1 = W^2$. Thus $\Phi$ is injective.
\end{proof}

\begin{proposition}
Assume that  Theorem \ref{mainspecialcase} is true.
Then for $d \gg e'$,
 the map $$\Phi: \DC^0_{\le (e,e')} (Y,X') \to \DC^0_{\le (e,de')} (X,X')$$ is surjective.
\end{proposition}

\begin{proof}
Let $W_X \subseteq X \times X'$ be an  dominant correspondence such that $\deg W_X \le (e, de')$, we will show that $W_X$ lies in the image of $\Phi$. It suffices to do so when $W_X$ is irreducible, so we will make that assumption.

 For a given $X'$, we fix once and for all two dominant morphisms $f_1, f_2: X' \to \PP^n$ such that $f_i^* \OO(1) = H'$ and  the product $f_1 \times f_2$ is a birational map between $X'$ and its image $Z \subseteq \PP^n \times \PP^n$. The maps  $f_1, f_2$ are assumed to be general. We denote the image of $W_X$ in $Y \times \PP^n \times \PP^n$ as $ W_Z$, it is birational to $W_X.$

     Then $(\id \times f_1)(W_X)$ and $(\id \times f_2)(W_X)$ are also irreducible dominant correspondences contained in $X \times \PP^n$. We may apply Theorem \ref{mainspecialcase} to conclude that  we have two irreducible dominant correspondences on $Y \times \PP^n$, $W^1$ and $W^2$ such that $\Phi(W^i) =  (\id \times f_i)(W_X)$. Let $W$ be the union of all components of the pull back  $W^1 \times _Y W^2$ that dominate $Y$ and contain the image of $W_Z$. We note that   $W$ is naturally a subvariety of $Y \times \PP^n \times \PP^n$. We claim that there is a component $W_0 \subseteq W$  contained in  $Y \times  Z$, such that  $\Phi(W_0) = W_Z$.  Assuming this claim, the proposition immediately follows since  we would obtain the equality, $\Phi((\id \times (f_1 \times f_2)^{-1}) (W_0)) = W_X$.

    We note now that we have not yet established that $W$ is  even nonempty, so let us do so. 

    Let $x \in X$ be a general point. We first note that the fibre over $x$ of the projection map $W^1 \to  \PP^n \times \PP^n $ is forced to be zero dimensional, if this were not the case the preimage of $X$ in $W^1$ would have dimension greater than or equal to that of $W^1$, this would contradict the irreducibility of $W^1$. The same goes for $W^2$. As a result, we may find  a Zariski open subvariety $U \subseteq Y$ such that the maps $W^1 \to Y$ and $W^2 \to Y$ have zero dimensional fibres over $U$ and furthermore $U \cap X$ is dense in $X$.

    Let $W^i_U$ denote the intersection $W^i \cap (U \times \PP^n).$  Let $W_U = W^1_U \times_U W^2_U$.
    We note that  the maps $W^i_U \to U$ are flat. As a result, $W_U \to U$ is flat and every component of $W_U$ dominates $U$.  We further note that $W_U$ contains $W_Z \cap (U \times \PP^n \times \PP^n)$.  

    We note that $W_U \subseteq W^1 \times_Y W^2$. By the above, there is  a component $V$ of $W_U$ containing a Zariski open subset of $W_Z$.Let $\bar V$ denote the closure of $V$ in $W^1 \times_Y W^2$.
    
    Then, $\bar V$ dominates $Y$ and contains $W_X$. Thus $\bar V$ defines a component of $W$ and hence $W$ is nonempty.

   We will now establish that the image of $W $ under the projection  map $W\to \PP^n \times \PP^n$  is $Z$.
   We note that $W $ is of dimension $n+1$ and that the image of $W$ contains $Z$. It suffices to establish that the image is equal to $Z$. To do this, it suffices to establish that for general $(z_1, z_2) \in Z$, the variety $W_{(z_1, z_2)} = W \cap (Y \times \{(z_1, z_2)\})$ is one- dimensional.
     Let $(z_1 , z_2) \in Z$ be general. Now $$W_{(z_1, z_2)} = (W_1 \cap (Y \times f_1^{-1}(z_1))) \cap (W_2 \cap (Y \times f_2^{-1}(z_2))),$$ here we view  the varieties  $W_{(z_1, z_2)}, (W_1 \cap Y \times f_1^{-1}(z_1)), $ and $ (W_2 \cap Y \times f_2^{-1}(z_2)) $ as subvarieties of $Y$.  Let $ C_1 = (W_1 \cap Y \times f_1^{-1}(z_1))$ and $C_2 = (W_2 \cap Y \times f_2^{-1}(z_2))$. Then $C_1, C_2$ are both curves of degree  $\le ke'$ in $Y$.    However $C_1$ and $C_2$ both contain $W \cap X  \times \{(z_1, z_2)\}\subseteq C_1, C_2 $ which is a Cayley Bacharach set for $K_X$ and thus contains at least $d -O(1)$ points , thus for $d \gg e'$, $C_1$ and $C_2$ must have a component in common. 

         This implies that $W$ must have a component $W_0$, such that $W_0$ dominates $\bar X'$ and maps to $\bar X'$, with fibers curves.
         We now  claim that $\Phi(W_0) = \bar W_X$. Since $\bar W_X$ is irreducible, it suffices to establish that $\Phi(W_0) \subseteq  \bar W_X$.

         Now $(\id \times f_i)(\Phi(W_0))= \Phi((\id \times f_i)(W_0)) \subseteq (\id \times f_i)(W_X)$. Thus, $(\id \times f_1 \times f_2)(\Phi(W_0)) \subseteq (\id \times  f_1 \times f_2)(W_X) .$ Since $f_1 \times f_2$ is birational this implies that $\Phi(W_0) \subseteq   W_Z$. This concludes the proof.
 
\end{proof}

\subsection{Injectivity}
In this section we will establish part of Theorem \ref{mainspecialcase},  we will establish that  in that context $\Phi$ is injective.

\begin{proposition}
    Suppose $W^1, W^2 \in \DC_{\le (e,e')}(Y,X')$ are dominant correspondences, such that $\Phi(W^1) = \Phi(W^2).$ Then if $d \gg e'$, we have the equality 
    $W^1 = W^2$.
\end{proposition} 
\begin{proof}
     It suffices to establish this in the case when $W^1$ and $W^2$ are irreducible, so we make that assumption. Let $W_0$ be a component of $\Phi(W^1)$.
     For $z \in X'$, we define $W^i_z =   W^i \cap (Y \times \{z\})$ and $W_{0,z} = W_0 \cap (X \times \{z\}).$

    Now we note that by Proposition 4.2 of \cite{Bast} for general $z \in X'$,  $W_{0,z}$ is a Cayley-Bacharach set for $K_X$.

       Further more,  for general $z \in X'$, $W^1_z,$ is a curve of degree $\le e$ and $W_{0,z} \subseteq W^1_z$. But by Theorem 1.13 of \cite{Ban} there is a unique minimal curve $C_0$ of degree $\le e'$ containing $W_{0,z}$, furthermore each component of $C_0$ contains more than $\approx d$ points of $W_{0,z}$.

  Now $W^1_z$ is a curve of degree $ \le e'$ such that 
 each component of $C_0$ intersects it at at least$\approx d$ many points, if $ d \gg (e')^2$ then this implies that $W^1_z$ contains each component of $C_0$ and hence  must contain $C_0$.  Similarly, $W^2_z$ must contain $C_0$. Thus $W^1 \cap W^2 \cap (Y \times \{z\})$ is of dimension $1$ and hence $W^1 \cap W^2$ is of dimension $n+1$. Since $W^1$ and $W^2$ are irreducible, this implies $W^1 =W^2$.     
\end{proof}
\subsection{Surjectivity}
In this section we will establish the other half of Theorem \ref{mainspecialcase}, we will show that the map $\Phi$ is surjective.
\begin{proposition}
    Let $n,d,d',e,e' \ge 1$. Let $Y$ be a smooth projective variety of dimension $n+1$, with a very ample divisor $H$. Let $X$ be a smooth hypersurface in $Y$ such that $X \in |dH|$. Let $X'$ bea smooth projective variety of dimension $n$, with a very ample divisor $H'.$
    Let $W_X \subseteq X \times X'$ be a dominant irreducible correspondence of degree $\le (e,de')$. Then if $d \gg e'$, there exists $W \subseteq Y \times X'$ such that $\Phi(W) = W_X$, i.e. the map $\Phi$ is surjective.
\end{proposition}
\begin{proof}
   For general $z \in X'$. $W_X \cap (X \times \{z\})$ is  a Cayley Bacharach set for $K_X$  (this is a result by Bastianelli in \cite{Bast}). By Theorem 1.13 of \cite{Ban} there is a unique curve of minimal degree $\le e$ containing it, say $C_z$. We note that for $d \gg 0$, $C_z \subseteq Y$. Let $f$ be the degree of this curve for general $z \in Z$.

        We note that $W_X$ defines a rational map $X' \dra \Hilb_f Y$ defined as $z \mapsto C_z$ . We may then construct a birational map  $\bar X' \to X'$ and a regular map $\bar X' \to \Hilb_f Y,$ resolving the above map. Let $\bar W_X$ denote the strict transform of $W_X$ in $Y \times \bar X'.$
    
    Let $\CCal \to \Hilb_f Y$ be the tautological family.
    Now we define $\bar W$ as the following pull back:

    \[
    \xymatrix{
     \bar W \ar[r] \ar[d] & \CCal \ar[d] \\
      \bar X' \ar[r]        &  \Hilb_f Y\\
    }
    \]
    Let $\bar W_0$ denote $\bar W$ minus all components that do not dominate $Y$ (all components have to dominate $\bar X'$ since $\bar W \to \bar X'$ is flat). 
     We define $W$ to be the image of  $\bar W_0$  in $Y \times X'$. Let us now confirm that $W$ has the required properties, i.e. that $W$ is indeed a dominant correspondence and that $\Phi (W) = W_X$.

         To do this  it suffices to establish that $\Phi(\bar W_0) = \bar W_X$.

    Let us first confirm that $\bar W_0$ is nonempty. Now,  $\bar W \to \bar X'$ is surjective  with one dimensional fibers by construction. The  map $\bar W \to Y$ is surjective, for the following reason- The image of $\bar W$ in $Y$ contains $X$. So either the image is all of $Y$ or it is the union of some proper subvarieties of $Y$ containing $X$. If we are in the latter situation, then $X$ is swept out by degree $\le e'$ curves which is a contradiction as the covering degree of $X$ is at least $\approx d$.  This establishes that $\bar W$ dominates $Y$ which in turn establishes that $\bar W_0$ is nonempty.

    We will now establish that $\bar W_0 $ contains $\bar W_X$. Note that  $\bar W$    contains $\bar W_X$. Let $V$ be a component of $\bar W$ that does not dominate $Y$. By  the construction of $\bar W$, $V$ is $n+1$ dimensional and the map $V \to \bar Z$ has 1 dimensional fibers. If $V$ contains $ \bar W_X$ then it must be the case that the image of $V$ under the projection is $X$. But then we must have that the curves $V \cap (X \times \{\bar z\})$ cover $X$. Since these curves are of degree $\le f$ and the covering degree of $X$ is at least $\approx d$ this cannot happen. Thus such a $V$ does not contain $\bar W_X$.
    We have thus established that $\Phi(\bar W_0)$  contains $\bar W_X$, we must now establish that it contains no other components.

Suppose that $\Phi(\bar W_0) =  \bar W_X + a_1W_X^1 \dots +a_rW_X^r$, where $W_X^r$ are all dominant correspondences.  It suffices to establish that all these $a_i$ are 0.
We note that by Proposition 4.2 of \cite{Bast}, for general $ z\in Z$ , $W_X^i \cap (X \times \{z\})$  is Cayley Bacharach for $K_X$ and as such must contain at least $d- O(1)$ points.  
Thus for general $z \in X'$, $\Phi(W) \cap (X \times \{z\})$ contains at least $(\deg(W_X) + a_1 \dots + a_r)d - O(1)$ many points.
However we can compute this quantity in another way- by looking at $C = W \cap (Y \times \{z\})$ and then intersecting $C$ with $X$, the intersection number is $d \deg_H C$. We thus have an equality-

$$d \deg_H C = (\deg(W_X) + a_1 \dots + a_r)d - O(1).$$ For $d$ large this implies

$\deg_H C = \deg W_X + \sum a_i$. But $\deg_H C \le \deg W_X$. This implies all the $a_i$ are equal to $0$.

This establishes that $\Phi(\bar W_0) = \bar W_X,$ which establishes the Proposition.
    \end{proof}

\begin{proposition}
    Let $W$ be as constructed above. Then $\deg W \le (e,e').$
\end{proposition}
\begin{proof}

By construction, the preimages of the map $W \to Z$ are degree $ \le e$ curves. The degree  of the map $W \to Y$ agree with that for $W_X \to X$ and that degree is $\le e'$.

\end{proof}

\begin{proof}[Proof of Theorem \ref{mainspecialcase}]
    The proof is immediate from the previous three propositions.
\end{proof}

\section{The operations $\Phi$ and $\Psi$ in families}

In this section we will describe how to do the operations $\Phi$ and $\Psi$ in families- this will be used  to prove Theorems \ref{mainpsi}, \ref{maink} and \ref{mainpsik}.

We begin with the following fairly general proposition which, loosely speaking, tells us that we can throw out nondominating components in families.
\begin{proposition}\label{famgen}
    Let  $B, T$ be   varieties. Let $X \to T$ be a variety. We will denote the fiber over a point $t \in T$ as $X_t$.
    Let $W \subseteq X$ be an irreducible subvariety. Let $D$ be a divisor in $X$. Let $f : D \to B$ be a dominant map, such that the restriction  of $f$ to $W \cap D$ is  dominant and that furthermore, the restriction of $f$ to  $ (W \cap D)_t = (W \cap X_t) \cap D$ is still dominant for any $t \in T$.
    Assume that $W \to T$ and $W \cap D \to T$ are flat families, with reduced fibers . Then there is an open subvariety $T_0 \subseteq T$  such that  there is a closed subvariety $W' \subseteq W \cap (D \times T_0)$  such that the fiber of $W'$ over $t\in T_0$ consists of all components of $(W \cap D)_t$ that dominate $B$.
\end{proposition}
Here is a diagram illustrating the situation of Proposition \ref{famgen}.

\[
\xymatrix{
W' \ar [r] & W \cap D \ar[d] \ar[r] & D \ar[d] \ar[r] & B\\
& W \ar[r] \ar[rd] & X \ar[d] & \\
& & T & \\
}
\]
\begin{proof}

    After possibly replacing $T$ by an open subset, we may assume that for each $t \in T $, the varieties $(W \cap D)_t$ have the same number of irreducible components. Over the generic point of $T$,  $W \cap D$ has some number of components, $W_1 , \dots W_k$, and on a closed point  $t\in T$ the fiber has the same number of components. Now we may simply take $W'$ to be the union of the closures of the $W_i$ that dominate $B \times T.$ It is easy to see that it has the required properties.
\end{proof}

\begin{proposition}\label{Psifam}
    Let $Y, Y'$ be smooth projective varieties. Let $T$ be a variety. Let $W \subseteq Y \times Y' \times T$ be a family of dominant $k$ correspondences, i.e. a subvariety of $Y \times Y' \times T$ whose fiber over every point is a dominant $k$ correspondence. Let $X' \subseteq Y'$ be a divisor in $|d'H'|$, $d'>>0$. Then there is an open subvariety $T_0 \subseteq T$ such that there exists a subvariety which we will denote $\Psi(W)$ such that $ \Psi(W) \subseteq Y \times X' \times T_0$ whose fiber over $t \in T_0$ is $\Psi(W_t).$
\end{proposition}
\begin{proof}
    This follows from applying Proposition \ref{famgen} to the family $W  \subseteq Y \times Y' \times T$, where the divisor $D$ is  $W \cap Y \times X' \times T$, the space $B$  is $Z$ and the map $f$ is the projection.
\end{proof}
\begin{proposition}\label{Phifam}
    Let $Y, X, X'$ be as in Theorem \ref{main} or \ref{maink} Let $W \subseteq Y \times X' \times T$ be a family of dominant $k$ correspondences over $T$. Then there is an open subvariety $T_0 \subseteq T$ such that there exists a subvariety $\Phi(W)$ satisfying $\Phi(W) \subseteq X \times X' \times T_0$ whose fiber over $t \in T_0$ is $\Phi(W_t).$
\end{proposition}
\begin{proof}
    The proof of this is similar to the preceding Proposition so we omit it.
\end{proof}

\begin{proposition}
    Let $Y,Y'$ be smooth projective varieties. Let $X' \subseteq Y'$ be a smooth projective divisor that is a member of $ |dH'|$. Let $V_t$ be a pencil of divisors with base $T$ corresponding to the linear system $|H|.$ Let $\mathcal{V} \to T$ denote the associated family over $T$. Given a family $W \subseteq V $, we can form the family $\Phi(W) \to T$ such that $\Phi(W)_t = \Phi(W_t)$.
\end{proposition}
\begin{proof}
    This follows by applying Proposition \ref{famgen} to the family $\mathcal{V} \subseteq  Y \times Y' \times T$, the divisor $ Y \times Z \times T $ and the map $f$ being the projection onto $Z$.
\end{proof}
An analogous Proposition for $\Psi$ is true, but we will not need to use it in the course of this paper.

\section{Proof of Theorem \ref{mainpsi}}
In this section we will prove Theorem \ref{mainpsi}. We will establish that  $\Psi$ is well defined, injective and surjective respectively.

We note that it is not even clear that for $W \in \DC^1(Y,Y')$, $\Psi(W)$ as previously defined is a dominant correspondence. We will establish this now. 
\begin{proposition} Assume $d' \gg e$.
    The map $\Psi$ defines a function from  $\DC^1_{\le (e,e')} (Y, Y')$ to $\DC^0 _{\le (e,de')}(Y,X')$. In particular for  $W \in \DC^1(Y,Y')$, $\Psi(W)$ is a dominant correspondence.
\end{proposition}
\begin{proof}
      Let $W \in \DC^1_{\le (e,e')} (Y, Y').$ Then $\Psi(W)$ is a well defined subvariety of $Y \times X'$ with appropriate bounds on degree, what remains to be established is that $\Psi(W)$ is a dominant correspondence.

      By the definition of $\Psi(W)$, each component of $\Psi(W)$ dominates $X'$, it remains to be established that these components dominate $Y$ as well.
      We first modify $W $ and $Y$ as follows:
    
    The map $W \to Y$ is generically fibered in curves of degree $f \le e$ contained in $Y'$. We thus get a rational map $Y \dra \Hilb_f Y'$.

     We may resolve this rational map to obtain a diagram 
\[
    \xymatrix{
\bar Y \ar[rd] \ar[d] & \\
Y \ar@{-->}[r] & \Hilb_f Y'\\
}
\]
We now let $\bar W$ denote the strict transform of $W$ under the birational map $\bar Y \times Y' \to Y \times Y'$. We observe that $\bar W$ is also the pullback of the universal family over $\Hilb_f Y'$  under the map $\bar Y \to \Hilb_f Y'$.
We now note that it suffices to prove the result for $\bar W$ instead of $W$. The main advantage that this affords us is that map $\bar W \to Y$ is fibered in curves as opposed to merely being generically fibered in curves. 
Suppose for the sake of contradiction that  $W_0 \subseteq \Psi(\bar W)$ is a component not dominating $Y$. Then it must be the case that the image of $W_0$ in $Y$ is a divisor say $D$ and furthermore the fibers of the map $W_0 \to D $ are curves of degree $\le f$ in $X'$.

 However, as a consequence each point in $W_0$ is covered by a curve of degree $\le f \le e$. This implies that the covering degree of $W_0 \le e$. However the covering degree of $W_0$ is greater than or equal to the covering degree of $X'$ which is at least $d' - O(1)$. Thus for $d' \gg e$, this is not possible. This establishes the result. 
\end{proof}

\begin{proposition}
    Let $e,e' \in \NN$. Then for $d' \gg e $, $$\Psi: \DC^1_{\le(e,e')}(Y, Y') \to \DC^0_{\le (d'e,e')}(Y, X')$$ is injective.
\end{proposition}
\begin{proof}
    Let $W^1, W^2 \in \DC^1_{\le (e,e')}(Y,Y')$ be such that $\Psi(W^1) = \Psi(W^2)$. It suffices to establish the result in the case when $W^1$ is irreducible, so we make that assumption.

     It suffices to establish that for a general hyperplane $V \in |H|$, $ W^1 \cap V \times Y' = W^2 \cap V \times Y'$. Furthermore it suffices to just  establish that the dimension of the intersection $W^1 \cap W^2 \cap V \times Y'$ is the dimension of $W^1 \cap V \times Y'$ or equivalently that $W^1 \cap V \times Y'$ and $W^2 \cap V \times Y'$ have a top dimensional component in common.

     We note that $V \times Y' \cap \Psi(W^1) = V \times Y' \cap \Psi(W^2)$, since $\Psi(W^1) = \Psi(W^2)$. However this then implies that the dominating components of $W^1 \cap (V \times X')$ agree with the dominating components of $W^2 \cap V \times X'$. But then, let $U_i$ denote the union of the dominating components of $ W^i \cap (V \times Y).$ We obtain that $\Phi(U_1) = \Phi(U_2)$ and thus by Theorem \ref{main}, $U_1 = U_2.$ As a result, $V \cap W^1$ and $V \cap W^2$ have a component in common,  which establishes the result.
\end{proof}
\begin{proposition}
 Let $e,e' \in \NN$. Then for $d' \gg e $, $$\Psi: \DC^1_{\le(e,e')}(Y, Y') \to \DC^0_{\le (d'e,e')}(Y, X')$$ is surjective.
\end{proposition}
\begin{proof}
    Let $W_{X'} \subseteq Y \times X'$ be a dominant correspondence- we will find a dominant 1 correspondence $W$, such that $\Psi(W) = W_{X'}$.  It suffices to deal with the case when $W_{X'}$ is irreducible so we make that assumption.
Let $V_t$  be a general pencil in $|H|$, let $L$ denote the base of the pencil. Let $\Hilb_{Y'} \to L$ denote the relative Hilbert scheme whose fiber over $t \in L$ is the Hilbert scheme of $V_t \times Y'$. Let $\Hilb^{\circ}_{Y'} \to L$ denote the open subscheme corresponding to dominant correspondences of degree $\le (e,e')$. 
Let $\Hilb_{X'} \to L$ denote the relative Hilbert scheme whose fiber over $t \in L$ is the Hilbert scheme of $V_t \times X'$. Let $\Hilb^{\circ}_{X'}$ denote the open subscheme corresponding to dominant correspondences of  degree $\le (d' e ,e')$.

Consider the family $\mathcal{W}_{X'} \to L$ whose fiber over $t$, $(\mathcal{W}_{X'})_t$ is   $W_{X'} \cap (V_t \times {X'})$. It is immediate that for general $t, $ the variety $(\mathcal{W}_{X'})_t$ dominates $V_t$, We also note that for general $x' \in X'$,  $W_{X'} \cap (Y \times \{x'\})$ is a curve and since $V_t$ is ample in $Y, $ the variety $W_{X'} \cap (V_t \times \{x'\})$ is nonemepty. Thus $(\mathcal{W}_{X'})_t$ dominates $X'$ as well.

We claim that for general $t$, each component of $(\mathcal{W}_{X'})_t$ dominates  $X'$. To see that this is the case,  we can use Proposition \ref{famgen} to find a family $\mathcal{W}^0_{X'}$ over $L_0 \subseteq L$ such that $({\mathcal{W}^0_{X'}})_t$ consists of all components of $({\mathcal{W}_{X'}})_t$ that dominate $X'$. But since $\mathcal{W}_{X'}$ is irreducible, it must be the case that for general $t,$ 

$$({\mathcal{W}^0_{X'}})_t = ({\mathcal{W}_{X'}})_t.$$
Thus for general $t$, all components of $({\mathcal{W}_{X'}})_t$ dominate $X'$. 

After replacing $L$ with an open subset, say $L_0$, we can use Proposition \ref{famgen} to find a flat family $U_{X'}$ over $L_0$ such that $(U_{X'})_t \subseteq ({\mathcal{W}_{X'}})_t $ are the components of $W_{X'} \cap (V_t \times X')$ dominating both $V_t$ and $X'$. This gives us an embedding of $L_0$ in $\Hilb^{\circ} _{X'}$.

Now in light of Proposition \ref{famgen} and Theorem \ref{main}, we have a stratification $S_{\alpha}$ of $(\Hilb^{\circ}_{Y'})_{red}$ such that on each stratum, we have an injective morphism $\Phi: S_{\alpha} \to (\Hilb^{\circ}_{X'})_{red}$ and furthermore, the union of the images of the $S_{\alpha}$ is $(\Hilb^{\circ}_{X'})_{red}$.

Let $S_{\alpha_0}$ be the stratum such that $\Phi(S_{\alpha_0})$ contains the general point of the image of $L_0$. After possibly replacing $L_0$ by an open subset, we may assume $\Phi(S_{\alpha_0})$ contains the image of $L_0$.  We note that $\Phi$ restricted to $L_0$ is injective and surjective  and hence birational, After replacing $L_0$ by an open subset we may assume that $\Phi$ is invertible.

We then may pullback the universal family via $\Phi^{-1}$ to $L_0$ to obtain a family $\mathcal{W} \to L_0$.   This results in a family $\mathcal{W} \to L_0$, such that for $t \in L_0,$ the fiber $\mathcal{W}_t$ satisfies $\Phi(\mathcal{W}_t) = (U_{X'})_t$. We now let $W$ be the image of $\mathcal{W}$ in $Y \times Y'$.

 We will now establish that $\Psi(W) = W_{X'}$. It suffices to show that $\Psi(W) \subseteq W_{X'}$. Suppose for the sake of contradiction that this is not the case. Then there must be some irreducible component of $\Psi(W),$ say $Z$ that is not contained in $W_{X'},$ and furthermore $Z$ must dominate both $Y$ and $X'$.
 Let $Z_t = Z \cap (V_t \times X')$ and let $\mathcal{Z} \to L_0$ be the associated family. For general $t, Z_t$ is not contained in $(W_{X'})_t$. Note that $Z_t$ dominates $V_t$. Furthermore  given general $x' \in X'$ $Z \cap (Y \times \{x'\})$ is a curve in $Y \times \{x'\}$. Since $V_t \times \{x'\}$ is an ample divisor, it must intersect this curve and hence $Z \cap (V_t \times \{x'\})$ is nonempty. Thus $Z_t$ dominates both $X'$ and $V_t$.

 We now claim that for general $t,$ each component of $Z_t$ must dominate $X'$. This uses the irreducibility of $Z$, we may apply Proposition \ref{famgen} to obtain a subfamily $\mathcal{Z}_0 \subseteq \mathcal{Z} \to L_0$ such that $\mathcal{Z_0}_t$ consists of those components of $Z_t$ that dominate $X'$. The irreducibility of $\mathcal{Z}$ then implies that for general $t$, $\mathcal{Z_0}_t = \mathcal{Z}_t = Z_t$, thus establishing the claim.

 As a result for general $t,$ there must be a component of $Z_t$ dominating both $X'$ and $V_t.$ However, then this component of $Z_t$ would be a component of $\Phi(W \cap  (V_t \times Y'))$ which by construction equals $ (W_{X'})_t$. But if $Z_t$ and  $(W_{X'})_t$ have a component in common for general $t$, the irreducibility of both $W_X'$ and $Z$ implies that they must be equal to each other. However we assumed that $Z$ is not contained in $W_{X'}.$ This is a contradiction. Thus no such $Z$ exists and $\Psi(W) = W_{X'}.$

    \end{proof}

    \section{Increasing $k$}
In this section we will prove Theorems \ref{maink} and \ref{mainpsik}.

Let $D = p^*H + (p')^* H'$. We note that $D$ is a very ample divisor on $Y \times Y'$.

We begin with a proposition.

\begin{proposition}
    Let $k \ge 1$, $W \in \DC^k(Y, Y')$ and $V$  a general hyperplane in $|D|$. Then $V \cap W \in \DC^{k-1}(Y, Y')$. 
\end{proposition}
\begin{proof}
    It suffices to establish  that the maps $W \cap V \to Y$ and $W \cap V \to Y'$ are still surjective. We will only establish that $W \cap V \to Y$ is surjective, since the same argument will work for the other map.

    Let $w \in W$ be a smooth point of $W$ such that the projection map is a submersion, let $y \in Y$ be its image.

    We may assume that $V$ intersects $W$ transversely at $w$.   It suffices to establish that the map on tangent space $T_w (V \cap W) \to T_yY$ is surjective. By transversality the tangent space $T_w (V \cap W) = T_w V \cap T_w W$ and furthermore the linear map $T_w (V \cap W) \to T_yY$ is merely the restriction of the original derivative map $T_w W \to T_y Y$, which is surjective by assumption.
     However since $|H \times H'|$ is very ample, we may choose $V$ appropriately so that $T_w (V \cap W)$ is an arbitrary codimension 1 subspace of $T_w W$. But  the restriction of the map to a general codimension 1 subspace is still surjective. Thus for general $V$, the map on tangent spaces is surjective and hence so is the entire map.

\end{proof}
We also have the following analogous proposition.

\begin{proposition}
     Let $W \in \DC^k(Y,X')$ and let $V$ be a general hyperplane in $|D|$. Then $V \cap W \in \DC^{k}(X, X')$. 
\end{proposition}
\begin{proof}
 This follows mutatis mutandis from the proof of the previous Proposition.
\end{proof}
\begin{proposition}\label{capPhi}
    Let $k\ge 1$. Assume Theorem \ref{maink}  is true for $k-1$. Let $W \in \DC^k_{\le (e,e')}(Y, X').$
    Then, for general $V \in |D|$, $$\Phi(V \cap W) = V \cap \Phi(W).$$
\end{proposition}
\begin{proof}
 It suffices to establish this result when $W$ is irreducible so we make that assumption. By Bertini's irreduciblity theorem we may also assume that  $V \cap W$ is irreducible. Furthermore $\Phi(W)$ is also irreducible by Theorem \ref{maink}  and so by Bertini's irreducibility theorem we may assume that $V \cap \Phi(W)$ is also irreducible. Since both $\Phi(V \cap W)$ and $V \cap \Phi(W)$ are irreducible varieties of the same dimension it suffices to establish that $\Phi(V \cap W) \subseteq V \cap \Phi(W)$  or equivalently $\Phi(V \cap W) \subseteq \Phi(W)$. Clearly $\Phi(V \cap W) \subseteq W \cap X \times X'$. Now $ W \cap X \times X' = \Phi(W) \cup W_1 \dots \cup W_r $ where the $W_i$ are components of $W \cap (X \times X')$ that do not dominate $X$. However, $\Phi(V \cap W)$ is an irreducible variety that dominates $X$ and thus it is forced to be contained in $\Phi(W)$. This concludes the proof.
\end{proof}
\begin{proposition}
    Let $k \ge 1$.
    Assume Theorem \ref{mainpsik} is true for $k-1$. Let $W \in \DC^k_{\le (e,e')}(Y, Y')$
    Then for general $V \in |D|$, $\Psi(V \cap W) = V \cap \Psi(W).$
\end{proposition}
\begin{proof}
    This follows mutatis mutandis from the proof of the preceding Proposition.
\end{proof}
We next establish that the map $\Phi$ is injective even when $k\ge 1$.
\begin{proposition}
   Let $d \gg e'$, let $ k\ge 0$. Then the map $$\Phi: \DC^k_{\le(e,e')}(Y, X') \to \DC^k_{\le (e,de')}(X, X')$$ is injective.
\end{proposition}
\begin{proof}
    Our proof will be by induction on $k$. If $k =0$, this follows by Theorem \ref{main} . Let $W^1 , W^2 \in \DC^k(Y,Z)$. Let us assume $\Phi(W^1) = \Phi(W^2)$. We must establish that $W^1 = W^2.$

     It suffices to establish that for a general hyperplane $V \in |D|$ that $V \cap W^1 = V \cap W^2$. Now for a general hyperplane $V$, both $V \cap W^1$ and $V\cap W^2$ are irreducible dominant correspondences and furthermore  $\Phi(V \cap W^i) =  V \cap \Phi( W^i)$ by Proposition \ref{capPhi}. However by induction, $\Phi$ is injective on lower degree correspondences. Thus, $V \cap W^1= V \cap W^2$ which is what we wanted to prove.
\end{proof}

\begin{proposition}
   Let $d' \gg e$, let $ k\ge 0$. Then the map $$\Psi: \DC^{k+1}_{\le(e,e')}(Y, Y') \to \DC^k_{\le (d'e,e')}(Y, X')$$ is injective.
\end{proposition}
\begin{proof}
    This follows mutatis mutandis from the proof of the previous proposition.
\end{proof}
\begin{proposition}
   Let $d \gg e'.$ Let $k \ge 0$. Then the map $$\Phi: \DC^{k}_{\le (e,e')}(Y,X') \to \DC^k(X, X') $$ is surjective and that furthermore irreducible correspondences are sent to irreducible correspondences.
\end{proposition}
\begin{proof}
    Our proof will be by induction on $k$, the case $k=0$ follows by Theorem \ref{main}.
    Assume the proposition is true for $k-1$.
    Let $W_X \in \DC^k(X, X')$. We will establish that there exists $W \in \DC^k$ such that $\Phi(W) = W_X.$ We may assume without loss of generality that $W_X$ is irreducible.
     Let $(V_t)_{t\in \PP^1} $ be a general pencil in $|D|$.  For general $t \in \PP^1$, $V_t \cap W_X $ is an irreducible dominant $k-1$ correspondence. 

      Let $L \subseteq \PP^1$ be a Zariski open such that for $t \in L,$  $V_t \cap W_X $ is an irreducible dominant correspondence. By the induction hypothesis, for each such $V_t \cap W_X$ there is a unique dominant correspondence $W_t \subseteq V_t\cap Y \times X'$ such that $\Phi(W_t) = V_t \cap W_X$.

      Let $\Hilb^{\circ}_{\le (e,e')}(Y \times X')$ denote the open subscheme of the Hilbert scheme of $Y \times X'$ parametrizing dominant $k$ correspondences of degree $\le (e,e')$. Let $\Hilb^{\circ}_{\le (e,de')} (X \times X')$ denote the open subscheme parametrizing dominant $k$ correspondences of degree $\le(e, de')$. Note that neither of the above two schemes is necessarily connected.

      Then, in light of Proposition \ref{Phifam} we have a stratification of $(\Hilb^{\circ}_{\le (e,e')}(Y \times X'))_{red}$ by subvarieties $S_{\alpha}$ such that on each stratum there is a map $\Phi:S_{\alpha} \to (\Hilb^{\circ}_{\le (e,de')} (X \times X') )_{red}$ realizing $\Phi$ on points.

      Due to our induction hypothesis, $\Phi : \coprod S_{\alpha} \to \Hilb_{\le (e,de')} (Y \times Z)$ is surjective and $\Phi$ is injective on each $S_{\alpha}$. Therefore, one can find a fixed stratum $\alpha_0$ such that the image of $S_{\alpha_0}$ contains $W_X \cap V_t $ for general $t \in L$. 
    
    This then implies that we have an open subset $L_0 \subseteq L$ lying in the image of $S_{\alpha}$,  such that the map $\Phi$ is invertible when restricted to the preimage of $L_0$ (this uses the fact that a bijective morphism, in this case the pullback of $\Phi$ to $L$, is birational).
    We then may pullback the universal family along $\Phi^{-1}$ to $L_0$, this gives us a family of generically irreducible varieties  $\mathcal{W} \to L_0$.
    We let $W$ be the image of $\mathcal{W}$ in $Y \times X'$, it is an irreducible correspondence. We now claim that $\Phi(W) = W_X$. This will establish the result. 

    We note that  by construction, $W \cap V_t$ contains $W_X \cap V_t$ and as a result, $W_X \subseteq W 
\cap X \times X' $ . Since  $W_X$ consists only of dominating correspondences, this implies that $W_X \subseteq\Phi(W)$ as well.

     We will now establish that $\Phi(W) \subseteq W_X$. Suppose for the sake of contradiction that $\Phi(W)$ contains another irreducible component $W_X'$. This $W_X'$ would have to be a dominating correspondence dominating both $X$ and $X'$.  
     However, we note that for $t \in L_0$, $\Phi(W_t) = V_t \cap W_X$ and thus $W'_X \cap V_t$ can not contain a component dominating both $X$ and $X'.$
    But for general $x \in X$, $W_X' \cap (\{x\} \times X')$ is a positive dimensional subvariety and must intersect $V_t$ since  since $V_t$  is an ample divisor, Thus,  $W'_X \cap V_t$ does in fact dominate $X$ (and $X'$  by the same argument).
     
    However this leads to the following contradiction: we may apply Proposition \ref{famgen} to obtain a subvariety $U \subseteq W'_X$, such that for general $t$, $U \cap V_t$ consists of all components of $W_X' \cap V_t$ dominating $X$. But since $W'_X$ is irreducible and $U$ is equidimensional with $W_X'$, for general $t$, each component of $W_X' \cap V_t$ must be contained in $U$ and must therefore dominate $X.$ By the same argument, for general $t$, each component of $W_X' \cap V_t$ must dominate $X'$ as well. But this is a contradiction. Thus no such $W_X'$ can exist and $\Phi(W) = W_X$.

\end{proof}

\begin{proposition}
   Let $d' \gg e.$ Let $k \ge 0$. Then the map $$\Psi: \DC^{k+1}_{\le (e,e')}(Y,Y') \to \DC^k(Y, X') $$ is surjective and takes irreducible correspondences to irreducible correspondences.
\end{proposition}
\begin{proof}
    This follows mutatis mutandis from the proof of the previous proposition.
\end{proof}

Thus we have established Theorems \ref{maink} and \ref{mainpsik}.
\section{Usage of smoothness}
In this section we will comment a little on the smoothness hypotheses on $X$ in Theorem \ref{main} and more generally to our other Theorems.  No generalisation to the case when $X$ is arbitrary is possible (though possibly we can relax the statement to normal with canonical singularities)- This is illustrated by the following example brought to our attention by Olivier Martin- take a high degree map from $\PP^1$ to $\PP^2$. Let $X$ be the image. Then if an analogue of Theorem \ref{main} held in this situation, the birational map $X \dra \PP^1$ would extend (as a rational map) to  all of $\PP^2$. However, this is not the case.

        This immediately prompts the question of where in the proof are we using smoothness. Essentially the main usage of smoothness is in all the steps where we argue that some set of points on a variety is Cayley- Bacharach for $K_X.$ For example the proof of \ref{domboth} heavily uses smoothness of $X$. This may not be apparent at first glance, as most of the argument that explicitly requires smoothness appears in the results that we cite from other authors, who in turn cite other authors.
         Hidden behind the lemmas we cite is Proposition 4.2 of \cite{Bast} which involves studying the trace map on  differentials coming from a correspondence. This requires that $X$ is smooth, though it is conceivable that one could get some version of this result with weaker conditions on $X$.

\section{What goes wrong in characteristic $p$?}
In this final section we will comment a little bit on what parts of our paper do not go through to characteristic $p$. Not only do our methods not work, we do not believe that the analogous results are themselves true in positive characteristic. The paper  \cite{S} shows that in positive characteristic result certain measures of irrationality of complete intersections behave a bit differently than in characteristic zero and it is not inconceivable that the same would be true for measures of association as well.

Primarily our main issue is in dealing with inseparable correspondences, in proving Theorem \ref{mainspecialcase} we use the fact that given a dominant correspondence $W \subseteq Y \times \PP^n$ and a general point $z \in \PP^n$, the set $W \cap X \times \{z\}$ is a Cayley Bacharach set for $K_X$. This uses the fact that the map $W \to \PP^n$ is separable, the proof of the above fact involves induced differentials. We do not know how to find an alternative construction in positive characteristic. 

We did consider the idea of looking at only separable dominant correspondences. The problem there was that $\Phi$ and $\Psi$ do not take separable correspondences to separable correspondences and thus it hard to isolate just separable correspondences.

\end{document}